\documentclass{article}
\usepackage[english]{babel}
\usepackage[letterpaper,top=2cm,bottom=2cm,left=3cm,right=3cm,marginparwidth=1.75cm]{geometry}

\usepackage{amsmath, amssymb, amsthm, xcolor}
\usepackage{graphicx}
\usepackage[colorlinks=true, allcolors=blue]{hyperref}

\newtheorem{definition}{Definition}[section]
\newtheorem{theorem}[definition]{Theorem}
\newtheorem{lemma}[definition]{Lemma}

\newtheorem{corollary}[definition]{Corollary}

\newtheorem{remark}[definition]{Remark}
\newtheorem{example}[definition]{Example}

\newtheorem{Proposition}[definition]{Proposition}
\DeclareMathOperator{\PG}{PG}
\DeclareMathOperator{\F}{\mathcal{F}}
\renewcommand{\P}{\mathcal{P}}
\renewcommand{\O}{\mathcal{O}}
\newcommand{\<}{\langle}
\renewcommand{\>}{\rangle} 

\newcommand{\cD}{{\mathcal D}}
\newcommand{\cF}{{\mathcal F}}

\newcommand{\cL}{{\mathcal L}}
\newcommand{\cO}{{\mathcal O}}
\newcommand{\cP}{{\mathcal P}}
\newcommand{\cQ}{{\mathcal Q}}
\newcommand{\cR}{{\mathcal R}}

\newcommand{\gauss}[2]{\genfrac{[}{]}{0pt}{}{#1}{#2}}
\newcommand{\qbin}[2]{\genfrac{[}{]}{0pt}{}{#1}{#2}}

\title{A common generalization of hypercube partitions and ovoids in polar spaces}
\author{Jozefien D'haeseleer, Ferdinand Ihringer  \& Kai-Uwe Schmidt} 

\begin{document}
\maketitle

\begin{abstract}
    We investigate what we call generalized ovoids, that is
    families of totally isotropic subspaces
    of finite classical polar spaces such that each maximal
    totally isotropic subspace contains precisely one
    member of that family. This is a generalization
    of ovoids in polar spaces as well as the natural
    $q$-analog of a subcube partition of the hypercube
    (which can be seen as a polar space with $q=1$).
    Our main result proves that a generalized ovoid of $k$-spaces
    in polar spaces of large rank does not exist.
    More precisely, for $q=p^h$, $p$
    prime, and some positive integer $k$, a generalized ovoid of $k$-spaces in a polar space
    $\cP$ with rank $r \geq r_0(k, p)$ in a vector space $V(n,q)$ does not exist.
\end{abstract}

\paragraph*{Dedication}
The following work is written in memory of Kai-Uwe Schmidt.
In March 2023 the first author, Jozefien D'haeseleer,
visited Kai-Uwe Schmidt in Paderborn for one week.
Mentioning the second author's recent preprint \cite{Irrsubcubepart}
as motivation, Kai suggested to investigate
the (very natural) topic of this preprint. That they did,
but both were occupied with other projects
after Jozefien had left Paderborn.
We hope that the present work provides an
execution of his idea which he would find
interesting to read.

\section{Introduction}

An ovoid of a polar space is a family of points $\cO$ such that
each generator (maximal totally isotropic subspace)
contains precisely one element of $\cO$.
The study of ovoids goes back to the geometric
construction of certain Suzuki groups by Tits \cite{Tits62}.
Ovoids in polar spaces were systematically defined
and studied by Thas \cite{Thas81}.
A \textit{generalized ovoid}, as introduced here,
is a family of totally isotropic subspaces $\cO$
such that each generator contains precisely one
element of $\cO$. This is the natural $q$-analog
of a subcube partition of a hypercube.
Let us sketch this connection in broad strokes
and in some greater detail later in Section \ref{sec:defs}.

A subcube partition is a partition of the 
hypercube $\{ 0, 1\}^n$ into subcubes.
Here we express subcubes as strings in 
$\{ 0, 1 , {*} \}$.
There are countless works written on them,
see the references within \cite{Irrsubcubepart}.
Following \cite{Irrsubcubepart}, we call 
a subcube partition \textit{irreducible}
if the only sub-partitions whose unions 
are a subcube are singletons and the entire partition.
We say that a subcube $s \in \{ 0, 1, {*}\}^n$ 
\textit{mentions} coordinate $i$ if $s_i \in \{ 0, 1 \}$.
A subcube partition $\cF$ is \textit{tight} if
it mentions all coordinates, that is
for each coordinate $i \in \{ 1, \ldots, n \}$
there exists an $s \in \cF$ such that $s_i \neq {*}$.
A subcube partition is called \textit{homogeneous}
if all its subcubes have the same dimension.
The main goal of \cite{Irrsubcubepart} is to determine
the minimum size of a tight irreducible subcube 
partition $\cF$ for any given $n$, 
but also other natural extremal questions
are investigated. More recently, but in the same vibe,
Alon and Balogh estimate the total number of partitions
of the hypercube in \cite{AB2024}.

The investigation in \cite{Irrsubcubepart} is mainly
motivated by complexity theory, see \cite{ComplexityPaper},
and was extended to hypercubes with larger alphabets
as well as linear subspaces instead of subcubes,
see also \cite{AffinePaper}. If we consider hypercubes
as distance-regular graphs with classical parameters,
cf.~Table 6.1 in \cite{BCN}, 
or as thin spherical buildings of one of 
the types $B_n/C_n/D_n$ restricted to \textit{generators} (maximal isotropic 
subspaces), then it is natural 
to also consider the generalization to dual polar graphs,
respectively, the graph of generators
of polar spaces. Note that here points of $\{ 0, 1\}^n$
correspond to generators of the polar space,
while points of the polar space correspond to 
$(n-1)$-dimensional subcubes of $\{ 0, 1 \}^n$.
Thus, we can study the $q$-analog of subcube partitions
in this setting. This is precisely the study
of generalized ovoids.

Our results are devided into two parts.
We give a limited number of constructions in Section \ref{sec:constr}.
Then in Section \ref{sec:nonex} we will prove our main result,
the asymptotic non-existence of generalized ovoids:

\begin{theorem}\label{thm:main_nonex}
    Let $p$ be a prime and let $k$ be a positive integer.
    Then there exists a constant $r_0(p, k)$ such that for all $r \geq r_0(p, k)$
    the following holds:
    For any positive integer $h$, put $q=p^h$.
    Let $\cP$ be a polar space of rank $r$
    over the field with $q$ elements.
    Then $\cP$ does not possess a family $\cO$
    of $k$-spaces such that each generator of $\cP$
    contains precisely one element of $\cO$.
\end{theorem}

Theorem \ref{thm:main_nonex}
is a generalization of a
classical result by Blokhuis and Moorhouse who observed
the following in Theorem 1.6 in \cite{BM1995}.

\begin{theorem}[Blokhuis \& Moorhouse (1995)] \label{thm:blokhouse}
    Let $p$ be a prime, $h$ a positive integer, and $q = p^h$.
    Let $\cO$ be a  partial ovoid of any finite classical polar space
    naturally embedded in a vector space of dimension $n$
    over the field with $q$ elements.
    Then
    \[
     |\cO| \leq \binom{p+n-2}{p-1}^h + 1 \leq (p+n-1)^{h(p-1)} + 1.
    \]
\end{theorem}

Subsequently, the result by Blokhuis and Moorhouse has been slightly improved by
Arslan and Sin, see \cite{AS2011}.
The rank of a polar space satisfies $n-2 \leq 2r \leq n$.
An ovoid of a rank $r$ polar space has size at least
$q^{r-1}+1 = p^{h(r-1)} + 1$. 
large, this is clearly less than $(p+n-1)^{h(p-1)} \leq (p+2r+1)^{h(p-1)}$.
Hence, for $k=0$, Theorem \ref{thm:main_nonex} is a special case.
We will provide a quantitative statement of Theorem \ref{thm:main_nonex}
in Section \ref{sec:nonex}. In the abstract we claim something
slightly stronger, namely that $\cO$ having subspaces
of dimension at most $k$ shows non-existence.
This will also follow from the
quantitative discussion in Section \ref{sec:nonex}.

\section{Polar Spaces}

For an extensive and detailed introduction about finite classical polar spaces, we refer to \cite{thashirschfeld}. We only repeat the necessary definitions and information. Note that in this article, we will work with algebraic dimensions, not projective dimensions. The subspaces of dimension $1$ (vector lines), $2$ (vector planes) and $3$ (vector solids) are called \emph{points}, \emph{lines} and \emph{planes}, respectively.
We denote the vector space of dimension $n$ over the field with $q$ elements by $V(n, q)$.

We start with the definition of finite classical polar spaces.
\begin{definition}
Finite classical polar spaces are incidence geometries consisting of subspaces that are totally isotropic with respect to a non-degenerate quadratic or non-degenerate reflexive sesquilinear form on a vector space $\mathbb{F}_q^{n}$. 
\end{definition}

A bilinear form for which all vectors are isotropic is called \emph{symplectic}; if $f(v,w)={f(w,v)}$ for all $v,w\in V$, then the bilinear form is called symmetric.
A sesquilinear form on $V$ is called \emph{Hermitian} if the corresponding field automorphism $\theta$ is an involution and $f(v,w)={f(w,v)}^{\theta}$ for all $v,w\in V$.
 We now list the finite classical polar spaces of rank $r$.


\begin{itemize}
 \item The hyperbolic quadric $\mathcal{Q}^{+}(2r-1,q)$ arises from a hyperbolic quadratic form on $V(2r,q)$. Its standard equation is $X_{0}X_{1}+\dots+X_{2r-2}X_{2r-1}=0$.
 \item The parabolic quadric $\mathcal{Q}(2r,q)$ arises from a parabolic quadratic form on $V(2r+1,q)$. Its standard equation is $X^{2}_{0}+X_{1}X_{2}+\dots+X_{2r-1}X_{2r}=0$.
 \item The elliptic quadric $\mathcal{Q}^{-}(2r+1,q)$ arises from an elliptic quadratic form on $V(2r+2,q)$. Its standard equation is $g(X_{0},X_{1})+\dots+X_{2r-2}X_{2r-1}+X_{2r}X_{2r+1}=0$ with $g$ a homogeneous irreducible quadratic polynomial over $\F_{q}$.
 \item The Hermitian polar space $\mathcal{H}(2r-1,q)$ (where $q$ is a square) arises from a Hermitian form on $V(2r,q)$, constructed using the field automorphism $x\mapsto x^{\sqrt{q}}$. Its standard equation is $X^{\sqrt{q}+1}_{0}+X^{\sqrt{q}+1}_{1}+\dots+X^{\sqrt{q}+1}_{2d-1}=0$.
 \item The Hermitian polar space $\mathcal{H}(2r,q)$ (where $q$ is square) arises from a Hermitian form on $V(2r+1,q)$, constructed using the field automorphism $x\mapsto x^{\sqrt{q}}$. Its standard equation is $X^{\sqrt{q}+1}_{0}+X^{\sqrt{q}+1}_{1}+\dots+X^{\sqrt{q}+1}_{2r}=0$.
 \item The symplectic polar space $\mathcal{W}(2r-1,q)$ arises from a symplectic form on $V(2r,q)$. For this symplectic form we can choose an appropriate basis $\{e_{1},\dots,e_{r},e'_{1},\dots,e'_{r}\}$ of $V(2r,q)$ such that $f(e_{i},e_{j})=f(e'_{i},e'_{j})=0$ and $f(e_{i},e'_{j})=\delta_{i,j}$, with $1\leq i,j\leq r$.
\end{itemize}

In this article all polar spaces we will handle are finite classical polar spaces.
We also give the definition of the rank and the parameter $e$ of a polar space.
\begin{definition}
The subspaces of maximal dimension (being $r$) of a polar space of rank $r$ are called \emph{generators}. We define the \emph{parameter} $e$ of a polar space $\mathcal{P}$ over $\F_{q}$ as the number $e$ such that the number of generators through an $(r-1)$-space of $\mathcal{P}$ equals $q^e+1$.
\end{definition}
The parameter of a polar space only depends on the type of the polar space and not on its rank. In Table \ref{tabele} we give the parameter $e$ of the polar spaces. 

\begin{table}[ht]\begin{center}
  \begin{tabular}{ | c | c| }
    \hline
    Polar space & $e$  \\ \hline \hline
    $Q^+(2r-1,q)$ & $0$  \\ \hline
    $H(2r-1,q)$ & $1/2$  \\ \hline
    $W(2r-1,q)$ & $1$  \\ \hline
    $Q(2r,q)$ & $1$  \\ \hline
    $H(2r,q)$ & $3/2$  \\ \hline
    $Q^-(2r+1,q)$ & $2$  \\ \hline
  \end{tabular}
  \caption{The parameter $e$.}\label{tabele}
\end{center}\end{table}

An important concept, associated to polar spaces, are polarities.
\begin{definition}
 A \emph{polarity} on $V(n,q)$ is an inclusion reversing involution $\perp$ acting on the subspaces of $V(n,q)$. In other words, $\perp^2$ is the identity, and any two subspaces $\pi$ and $\sigma$ satisfy $\pi \subseteq \sigma\Leftrightarrow\rho^\perp \subseteq \pi^\sigma$.
\end{definition}

Consider a non-degenerate sesquilinear form $f$ on the vector space $V=V(n,q)$, or the bilinear form $f$, based on a non-degenerate quadratic form $Q$ on the vector space $V=V(n,q)$, with $f(v,w)=Q(v+w)-Q(v)-Q(w)$. For a subspace $W$ of $V$, we can define its orthogonal complement with respect to $f$:
\begin{align*}
 W^\perp=\{v\in V \ | \ \forall w\in W: f(v,w)=0\}.
\end{align*}
The  map $\perp$ that maps the subspace $W$ onto the  subspace $W^\perp$, is a polarity, and every polarity arises in this way.
To every (finite classical) polar space a polarity is associated (but not the other way around).
The image of a subspace $\pi$ with dimension $t$ on the polar space $\mathcal{P}$ of rank $r$ under the corresponding polarity is its \emph{tangent space} $T_{\pi}(\mathcal{P})$, which is the subspace spanned by the $(t+1)$-spaces through $\pi$ such that they are contained in the polar space, or meet the polar space in $\pi$.
Moreover, note that $T_{\pi}(\cP)\cap \cP$ is a cone with vertex $\pi$ and with basis a polar space $\cP'$ of the same type as $\cP$, and with rank $r-t$.
\medskip

We will work with the \emph{Gaussian binomial coefficient} $\begin{bmatrix}a\\b \end{bmatrix}_q$ for positive integers $a,b$ and $q\geq 2$:

\begin{align*}
\begin{bmatrix}a\\b \end{bmatrix}_q=\prod_{i=1}^b \frac{q^{a-b+i}-1}{q^i-1} = \frac{(q^a-1)\dots (q^{a-b+1}-1)}{(q^b-1)\dots (q-1)}.
\end{align*}
We write $\qbin ab$ if the field size $q$ is clear from the context. The number $\qbin ab_q$ equals the number of $b$-spaces in $V(a,q)$, and the equality $\qbin{a}{b}_q = \qbin{a}{a-b}_q$ follows immediately from duality.\\

\begin{lemma}[{\cite[Lemma 9.4.1]{BCN}}]\label{lem:numberofkinpolar}
The number of $k$-spaces in a finite classical polar space $\mathcal{F}$ of rank
	$r$ and with parameter $e$, embedded in a vector space over the field $\mathbb{F}_q$,  
	is given by $$ \qbin{r}{k} \prod_{i=1}^{k} (q^{r+e-i}+1). $$
		Hence, the number of points in $\mathcal{F}$ is $\qbin{r}{1}(q^{r+e-1}+1)$. The number of
		 generators in $\mathcal{F}$ is $\prod_{i=1}^{r} (q^{r+e-i}+1)$.
\end{lemma} 

\section{Generalized Ovoids in Polar Spaces} \label{sec:defs}

In the following we formally define generalized ovoids for
polar spaces. First note that polar spaces are a $q$-analog
of the hypercube. For this, the following table for translating
between $\{ 0, 1 \}^r$ and a polar space of rank $r$ is helpful.

{
\medskip 
\begin{tabular}{l|l}
    \bf Hypercube                & \bf Polar Space \\ \hline 
    dimension $r$                & rank $r$ \\
    point $x \in \{ 0, 1 \}^r$   & generator \\
    subcube of dimension $r-1$   & point \\
    subcube of dimension $r-2$   & line \\
    subcube of dimension $r-d$   & subspace of rank $d$\\
    union of subcubes $A_i$      & union of all generators which contain one $A_i$\\ \hline
    $2^r$ points  & $\prod_{i=1}^r (q^{i+e-1}+1)$ generators\\
    $2r$ subcubes of dimension $r-1$  & $(q^{r+e-1}+1) \gauss{r}{1}$ points
\end{tabular}
}

\medskip

Note that for $q=1$ the last two counts are the same in both columns,
illustrating how the word $q$-analog is justified.
Of course this is well-known: In terms of distance-regular graphs,
the known families of distance-regular graphs with classical parameters $(r, q, 1, q^e+1)$
are the hypercubes for $q=1$. Furthermore, they are the dual polar graphs for $q$ a prime power, which
are pseudo $D_r(q)$ graphs, see \cite{BCN}. In terms of diagram geometry or
building theory, the Coxeter-Dynkin diagram of a hypercube is $B_r = C_r$.

\begin{definition}
    A \textit{partial generalized ovoid} $\cO$ of a polar space $\cP$ is
    a set of totally isotropic (nontrivial) subspaces of $\cP$ such that each
    generator contains at most one element of $\cO$.
\end{definition}

\begin{definition}
    A \textit{generalized ovoid} $\cO$ of a polar space $\cP$ is 
    a set of totally isotropic (nontrivial) subspaces of $\cP$ such that each 
    generator contains precisely one element of $\cO$.
\end{definition}

Note that unlike ovoids, generalized ovoids always exist, for instance
the set of all generators is a generalized ovoid. For consistency
with the definitions in \cite{Irrsubcubepart}, one might want to
allow $|\cO|=1$ with $\cO$'s only element being the trivial subspace
(as each subspace is incident with the trivial subspace),
but here we will exclude it.

\begin{definition}
   A generalized ovoid $\cO$ is \textit{reducible} 
   if there exists a subset $\cO' \subseteq \cO$,
   where $|\cO'| \geq 2$, such that the union 
   of all generators which contain one element
   of $\cO'$, have a non-trivial subspace $\pi$
   as their intersection.
\end{definition}

That is, if $\cO$ is reducible, then we can replace the 
subspaces in $\cO'$ by $\pi$ and obtain a smaller generalized ovoid.

\begin{example}
    Suppose that a polar space $\cP$ possesses an ovoid $\cO$.
    Replace one point $P$ of $\cO$ by a set of 
    lines $\cL$ of $\cP$ such that $\cL$
    corresponds to an ovoid of $\cP$ in the quotient
    of $P$. Then the new generalized ovoid is reducible. 
\end{example}

\subsection{Homogeneous Generalized Ovoids}

A particular case occurs when all the elements 
of a generalized ovoid have the same dimension.
Particularly, classical ovoids have this property.

\begin{definition}
    A generalized ovoid $\cO$ of a rank $r$ polar space  is
    \textit{homogeneous} if all its 
    elements have the same (algebraic) dimension $k$. In this case
    we call $\cO$ an $(r,k)$-ovoid.
\end{definition}

\begin{lemma}\label{lem:nrinovoid}
    Let $\O$ be an $(r,k)$-ovoid in a non-degenerate polar space $\mathcal{P}$ of rank $r$ and type $e$, then $|\O| = \prod_{i=1}^k (q^{r+e-i}+1)$.
\end{lemma}
\begin{proof}
    Since there are $\prod_{i=1}^{r}(q^{r+e-i}+1)$ generators in $\P$, 
    and $\prod_{i=1}^{r-k}(q^{r-k+e-i}+1)$ generators through a
    $k$-space in $\P$, we have that the lemma follows from a 
    double counting of the couples 
    $\{(\alpha, \pi) \,|\, \alpha \in \O, \alpha \subset \pi, \pi\in \mathcal{P}, \dim(\pi)=r\}$.
\end{proof}

For hypercubes, a notion of tightness  is necessary as otherwise 
subcubes of arbitrary small size exist. In particular,
small dimensional examples might hide in high dimension. For 
polar spaces, this is not the case. Let us briefly justify this.
For example, a partition of a hypercube $\{0, 1\}^r$ into subcubes of dimension
$r-k$ has always size $2^k$. For a polar space, a set of totally
isotropic $k$-spaces such that each generator contains precisely
one $k$-space, has size $\prod_{i=1}^k (q^{r+e-i}+1)$ for 
some $e \in \{ 0, 1/2, 1, 3/2, 2 \}$. Note that this corresponds 
to $2^k$ and is independent of $r$ for $q=1$, while for $q$ 
a prime power the definition of the size depends on $r$.

\begin{lemma}\label{lemmaquotient}
If there exist an $(r,k)$-ovoid in $\cP$ with rank $r$ and parameter $e$, then there exist an $(r-1,k)$-ovoid in a polar space $\cP'$ of rank $r-1$ and parameter $e$.
\end{lemma}
\begin{proof}
Let $\cO$ be an $(r,k)$-ovoid in $\cP$ and let $P$ be a point in $\cP$, not contained in an element of $\cO$. Consider the tangent hyperplane $T_P(\cP)$ of $P$. We know that $T_P \cap \cP$ is the cone $\< P, \cP' \>$ with vertex the point $P$ and basis the polar space $\cP'$ of the same type as $\mathcal{P}$, but with rank $r-1$. 
For every element $\pi\in\cO\cap T_P(\cP)$, let $\pi'$ be the subspace $\langle P, \pi\rangle \cap \mathcal{P'}$, and let $\cO' = \{\pi' \,|\, \pi\in \cO\}$. 
As we know that every generator through $P$ contains an element of $\cO$, it follows that $\cO'$ is an $(r-1,k)$-ovoid in $\cP'$.
\end{proof}

\section{Examples for Generalized Ovoids} \label{sec:constr}

\subsection{Non-Homogenous Examples}

Our main concern in the non-homogenous case is the minimum size of a generalized ovoid in a given polar space. We denote the \emph{type} of a generalized ovoid $\cO$ as a sequence $1^{n_1}2^{n_2}3^{n_3}\dots r^{n_r}$ if $\cO$ consists of $n_i$ subspaces of dimension $i$.

Here a small table of the smallest generalized ovoids
and their type is given, for small rank $3$ polar spaces.
We do not include $Q^+(5, q)$ as there ovoids exist.

{
\medskip
\begin{tabular}{l|rr}
    Polar Space & Size & Type \\ \hline
    $W(5, 2)$ & $21$ & $1^6 2^{15}$ \\
    $Q^-(7, 2)$ & $\geq 153$ & $2^{153}$ \\
    $W(5, 3)$ & $\geq 232$ & $1^{12} 2^{140} 3^{80}$ \\
\end{tabular}
}

The homgenous example for $Q^-(7, 2)$ is described in \S\ref{sec:msystem_ex}.

\subsection{Recursive construction}
We start with the following lemma, which follows immediately from the proof of Lemma \ref{lemmaquotient}.
\begin{lemma}
    Let $\cO$ be an $(r, k)$-ovoid in a polar space $\cP$,
    where $r > k \geq 2$.
    Then there exists an $(r-1, k)$-ovoid in the quotient of a point of $\cP$.
\end{lemma}

\begin{lemma}
    Let $\mathcal{O}$ be an $(r,l)$-ovoid in a non-degenerate polar space $\mathcal{P}$ of rank $r$ and parameter $e$, and let $\mathcal{O}'$ be an $(r-l, k)$-ovoid in a polar space $\mathcal{P'}$ of rank $r-l$ and parameter $e$. Then there exist an $(r, k+l)$-ovoid in $\mathcal{P}$.
%
\end{lemma}
\begin{proof}
    Let $\mathcal{O}$ be an $(r,l)$-ovoid in $\mathcal{P}$. For every $l$-space $\alpha$ in $\O$, we consider its tangent space $T_{\alpha}(\mathcal{P})$, which is a cone with vertex $\alpha$, and basis a non-degenerate polar space $\mathcal{P'}$ of rank $r-l$ and type $e$. In this polar space $\mathcal{P'}$, we take an $(r-l, k)$-ovoid $\O'$, and let $\O_{\alpha}$ the set of all $(k+l)$-spaces $\{\langle \alpha, \tau\rangle |\tau \in \O'\}$ in $\mathcal{P}$. 
    Now we prove that the set $\O'' = \bigcup_{\alpha\in\O} \O_{\alpha} $ of $(k+l)$-spaces is an $(r, k+l)$-ovoid in $\mathcal{P}$.
    First note that $|\O''| = |\O|\cdot |\O_{\alpha}|$. By Lemma \ref{lem:nrinovoid}, we see that $\O''$ has the size of an $(r, k+l)$-ovoid, and hence, it is sufficient to prove that every generator of $\P$ contains at most one element of $\O''$. Suppose there is a generator $\pi$ containing two elements $\beta_1, \beta_2$ of $\O''$. Let $\alpha_\pi$ be the (unique) element of $\O$ contained in $\pi$. By the construction of $\O''$, we know that $\alpha_{\pi}\subset \beta_1\cap \beta_2$. Let $\P_{\alpha}$ a polar space of rank $r-l$ in the quotient of $\alpha_{\pi}$. Then we know that $\beta_1\cap \P_{\alpha}$ and $\beta_2\cap \P_{\alpha}$ are contained in $\O_{\alpha}$, which gives a contradiction, since the generator $\pi\cap \P_{\alpha}$ contains two elements of $\O_{\alpha}$.
\end{proof}
\begin{remark}
\begin{enumerate}
    \item Note that the $(r,k+l)$-ovoid constructed in the previous lemma is reducible, since the set of all generators containing an element of $\cO_{\alpha}$ contains the subspace $\alpha$.
    \item We can generalize this construction for non-homogeneous ovoids:

    Let $\mathcal{P}$ be a polar space, and let $\mathcal{O}$ be a generalized ovoid. Then, for every element $\pi$ in $\mathcal{O}$, we investigate its quotient space, and in this quotient space, we take a generalized ovoid $\mathcal{O}_{\pi}$. Now, Let $\mathcal{F}_{\pi}=\{\langle \pi, \tau\rangle | \tau \in \mathcal{O}_\pi\}$. Then $\bigcup_{\pi\in \mathcal{O}} \mathcal{F}_{\pi}$ is another generalized ovoid in $\mathcal{P}$.
\end{enumerate}  
\end{remark}

\subsection{Examples of \texorpdfstring{$(r,r-1)$}{(r, r-1)}-ovoids in \texorpdfstring{$Q^+(2r-1, q)$, $H(2r-1, q)$, and $Q(2r, q)$}%
{Q+(2r-1, q), H(2r-1,q), and Q(2r, q)}}

Let $\cP$ be one of $Q^+(2r-1, q)$, $H(2r-1, q)$, and $Q(2r, q)$, respectively,
with corresponding parameter $e \in \{ 0, 1/2, 1 \}$.
Let $\cP'$ be one of $Q(2r-2, q)$, $H(2r-2, q)$, and $Q^-(2r-1, q)$, respectively.
Then the generators of $\cP'$ have rank $r-1$ and each generator
of $\cP$ contains precisely one generator of $\cP'$.
Hence, $\cP'$ is an $(r, r-1)$-ovoid.

\begin{Proposition}
  The number of pairwise non-isomorphic $(r, r-1)$-ovoids in $\cP$ is
  at least the number of pairwise non-isomorphic partial ovoids
  with at most $X \leq q^{r+e-1}/5$
  elements in $\cP'$.
\end{Proposition}
\begin{proof}
    We know that the generators of $\cP'$ form an $(r, r-1)$-ovoid $\cO$.
    Let $\cR$ be a partial ovoid of $\cP'$.
    Construct a new $(r, r-1)$-ovoid $\cO'$ by
    repeating the following for each point $P$ in $\cR$:

    Consider the quotient space $\cQ$ of $P$.
    This is a polar space of the same type as $\cP$
    and rank $r-1$. In $\cQ$, the elements of $\cP'$ through
    $P$ correspond to a polar space $\cQ'$ of the same
    type as $\cP'$ and rank $r-2$.
    Let $\cQ''$ in $\cQ$ be isomorphic with $\cQ'$,
    but with $\cQ' \neq \cQ''$. Replace
    all $(r-1)$-spaces $S$ through $P$ with $S/P \in \cQ'$
    by all $(r-1)$-spaces $S$ through $P$ with $S/P \in \cQ''$.

    The resulting set $\cO'$ is still an $(r, r-1)$-ovoid
    as each generator through $P$ contains precisely one of the
    generators of $\cQ''$. The fact that $\cR$ is a
    partial ovoid guarantees that we can do this independently
    for all $P$ in $\cR$.

    For two non-isomorphic choices of $\cR$,
    the resulting $(r, r-1)$-ovoids must be non-isomorphic. First note that for each point in the partial ovoid $\cR$ in $\cP'$ we remove at most $\prod_{i=1}^{r-2}(q^{r-1+e-i}+1)-\prod_{i=1}^{r-2}(q^{r-2+e-i}+1)$ elements of $\cP'$, which is the number of generators in the polar space $\cQ'$ 
    of rank $r-2$ and parameter $e+1$, minus the number of generators in the polar space $\cP'' = \cQ'\cap \cQ''$ of rank $r-2$ and parameter $e$, see Lemma \ref{lem:numberofkinpolar}. Since the partial ovoid of $\cP'$ has size at most $X$, we find that the
    procedure above removes at most
    \begin{align*}
        \frac{X\left(\prod_{i=1}^{r-2}(q^{r-1+e-i}+1)-\prod_{i=1}^{r-2}(q^{r-2+e-i}+1)\right)}{\prod_{i=1}^{r-1}(q^{r+e-i}+1)} = \frac{X(q^{r-2}-1)q^e}{(q^{r+e-2}+1)(q^{r+e-1}+1)}
    \end{align*}
    of the generators of $\cP'$ from $\cO$. 
    For $X \leq q^{r+e-1}/5$ this
    fraction is at most
    $\frac{1}{5}$. Hence, the hyperplane containing 
    $\cP'$ is the unique hyperplane of the ambient vector space 
    which contains at least $\frac45$ of the elements of $\cO$.
    Hence, we can
    reconstruct $\cP'$ from $\cO'$. Hence,
    we can reconstruct $\cR$ from $\cO'$.
    Hence, non-isomorphic partial ovoids
    of $\cP'$ yield non-isomorphic $(r, r-1)$-ovoids.
\end{proof}

\subsection{Examples of \texorpdfstring{$(r,r-1)$}{(r, r-1)}-ovoids in \texorpdfstring{$Q^+(2r-1, q)$}{Q+(2r-1, q)}} \label{sec:Qplus_ex}

We want to find a set $Y$ of comaximal subspaces of $Q^+(2r-1, q)$ 
such that each maximal subspace contains precisely one element of $Y$.
As there are two types of maximal subspaces and each comaximal
subspace lies in one of each type, we are simply asking for a perfect 
matching in the (bipartite) graph of maximal subspaces of $Q^+(2r-1, q)$,
two adjacent if they meet in a comaximal subspace.
The graph has $2v := 2 \prod_{i=1}^{r-1} (q^i+1)$ vertices 
and degree $k := (q^r-1)/(q-1)$. It is easy to say that such perfect matchings
exist using Hall's marriage theorem. More precisely,
a result by Schrijver \cite{Schrijver1998} shows that a bipartite graph on $2v$ vertices 
and degree $k$ has at least 
\[
 \left( \frac{(k-1)^{k-1}}{k^{k-2}} \right)^v
\]
perfect matchings. In our case this is at least 
\[
    \left(\frac{k^2}{k-1}\left(1-\frac{1}{k}\right)^k\right)^v\geq \left(\frac{k^2}{k-1} \cdot \frac{1}{e+1}\right)^v \geq \left(\frac{q^{(r-1)} }{e+1} \right)^v\geq \left(\frac{q^{(r-1)} }{e+1} \right)^{q^{\binom{r}{2}}}.
\]
Hence, it is clear that $(r, r-1)$-ovoids are plentiful
and a classification is impossible. It is clear that
almost all of these $(r, r-1)$-ovoids are not contained in a hyperplane
(as there are far fewer hyperplanes). We also believe
that almost all of them are irreducible, but we lack 
a proof.


\subsection{Example of \texorpdfstring{$(3,2)$}{(3, 2)}-ovoid in \texorpdfstring{$Q^+(5,q)$}{Q+(5, q)}}
Let $Q=Q^+(5,q)$ be the non-degenerate hyperbolic quadric in $\PG(5,q)$, with polarity $\perp$, and let $\ell$ be a line in $\PG(5,q)$, disjoint from $Q$. 
It is known that $\ell^\perp\cap Q$ is a non-degenerate elliptic quadric $Q_3 = Q^-(3,q)$.

\begin{lemma}\label{lem:planeintersectQ}
    Let $P$ be a point in $\ell^\perp$. If $P\in Q_3$, then $\langle P,\ell\rangle \cap Q = \{P\}$, and if  $P\notin Q_3$, then $\langle P,\ell\rangle \cap Q$ is a conic $Q(2,q)$.
\end{lemma}
\begin{proof}
    If $P\in Q$, then $\langle P,\ell\rangle$ is a plane contained in the tangent hyperplane of $P$, and containing a line $\ell$ disjoint from $Q$. This implies that $\langle P,\ell\rangle$ does not contain lines, and hence, $P$ is the only point from $Q$ contained in it. 
    
    If $P\notin Q$, then $\langle P,\ell\rangle$ is a plane not contained in a tangent hyperplane of a certain point $P'$, as otherwise, $P'\in \ell^\perp$. Hence, $\langle P,\ell\rangle\cap Q$ is a non-degenerate conic $Q(2,q)$.
\end{proof}
Now we investigate how a solid $\sigma$ through $\ell$ can intersect $Q$. Note that the only possibilities for the intersection $\sigma\cap Q$ are a $Q^-(3,q), Q^+(3,q)$ or the cone $PQ(2,q)$, as we know that it should contain a line $\ell$, disjoint from $Q$.
\begin{lemma}
    Let $m$ be a line in $\ell^\perp$.
    \begin{enumerate}
        \item If $m\cap Q = \{P\}$, then $\langle m,\ell\rangle\cap Q = PQ(2,q)$.
        \item If $m\cap Q = \{P_1,P_2\}$, then $\langle m,\ell\rangle\cap Q = Q^-(3,q)$.
        \item If $m\cap Q = \emptyset$, then $\langle m,\ell\rangle\cap Q = Q^+(3,q)$.
    \end{enumerate}
\end{lemma}
\begin{proof}
    \begin{enumerate}
        \item If $m\cap Q = \{P\}$, then $\langle m,\ell\rangle$ is contained in the tangent hyperplane $T_P(Q)$. If $\langle m,\ell\rangle$ would contain a plane of $Q$, then $\ell$ cannot be disjoint from this plane, and hence, disjoint from $Q$. This implies that $\langle m,\ell\rangle$ does not contain planes of $Q$, and hence, it should intersect $Q$ in the cone $PQ(2,q)$.
        \item If $m\cap Q = \{P_1,P_2\}$, then, by Lemma \ref{lem:planeintersectQ}, we know that $\langle m,\ell\rangle$ contains two planes $\pi_1$ and $\pi_2$ such that $\pi_i\cap Q=P_i$, and furthermore $\langle m,\ell\rangle$ is not contained in the tangent hyperplanes $T_{P_1}(Q)$ nor $T_{P_2}(Q)$. Hence $\langle m,\ell\rangle\cap Q = Q^-(3,q)$.
        \item If $m\cap Q = \emptyset$, then, by Lemma \ref{lem:planeintersectQ}, we know that all planes through $\ell$ meet $Q$ in a conic. Hence, $\langle m,\ell\rangle\cap Q = Q^+(3,q)$. \qedhere
    \end{enumerate}
\end{proof}
Now we take a line spread $S$ in $\ell^\perp$. Let $\alpha$ be the number of tangent lines to $Q$ in $S$. Since we know that $|S|=q^2+1$ and $S$ partitions the $q^2+1$ points of $Q_3$, we can check that the number of bisecants in $S$ is equal to the number of lines disjoint to $Q$ in $S$, which is $\frac{q^2+1-\alpha}{2}$.

For every line $m\in S$, let $\mathcal{F}_m$ be the set of lines of $Q$ in $\langle m, \ell\rangle$. Note that $\mathcal{F}_m$ contains the $2q+2$ lines of $Q^+(3,q)$ if $m\cap Q = \emptyset$, that $\mathcal{F}_m$ contains the $q+1$ lines of a cone $PQ(2,q)$ if $|m\cap Q|=1$ and $|\mathcal{F}_m| = 0$ if $|m\cap Q|=2$.  

\begin{theorem}
Let $\mathcal{F} = \bigcup_{m\in S}\mathcal{F}_m$. Then $\mathcal{F}$ is a $(3,2)$-ovoid in $Q$.
\end{theorem}
\begin{proof}
    We have to prove that every plane in $Q$ contains precisely one line of $\mathcal{F}$.
    First note that $|\mathcal{F}| = 2(q+1)\cdot \frac{q^2+1-\alpha}{2}+(q+1)\cdot \alpha = q^3+q^2+q+1$. As we know that a $(3,2)$-ovoid in $Q^+(5,q)$ contains this number of lines, it is sufficient to prove that every plane in $Q$ contains at most one line in $\mathcal{F}$.
    Suppose there is a plane $\pi$ containing two lines $l_1, l_2$ of $\F$. Then $l_1$ and $l_2$ intersect in a point, and hence, $\pi$ should be contained in one of the solids $\langle m, \ell\rangle$ for $m\in S$. But then, $\pi\cap \ell\neq \emptyset$, which gives a contradiction, since $\ell$ is disjoint from $Q$.
\end{proof}

\subsection{Example of \texorpdfstring{$(3, 2)$}{(3, 2)}-ovoid in \texorpdfstring{$Q^-(7, 2)$}{Q-(7, 2)}} \label{sec:msystem_ex}
A $m$-system of a polar space $\cP$ is a family $M$ of $(m+1)$-spaces
of $\cP$ such that $S^\perp \cap T$ is trivial for all distinct $S, T \in M$.
See \cite{ShultThas1994}.
Let $X$ be the point set of the classical $1$-system of
$Q^-(7, q)$ which can be obtained by field reduction from $Q^-(3, q^2)$.
Then lines of $Q^-(7, q)$ meet $X$ in $0$, $1$, $2$, or $q+1$ points.
Then there are
\begin{itemize}
 \item $q^4+1$ lines in $X$,
 \item $(q^4+1)(q+1) q^4/2$ secants,
 \item $(q^4+1)(q+1)(q^3+q)$ tangents,
 \item $(q^4+1)(q^5-q^4)/2$ passants.
\end{itemize}
For $q=2$, $(q^5-q^4)/2 = q^3$ and taking the union of lines in $X$
and all passants has size $(q^3+1)(q^4+1) = 153$. Indeed, this is
a $(3, 2)$-ovoid of $Q^-(7, 2)$. 
We could not generalize this constrution.


\section{The Non-Existence of \texorpdfstring{$(r, k)$}{(r, k)}-Ovoids for \texorpdfstring{$r \gg k$}{r >> k}} \label{sec:nonex}

Here we will show the following quantitative version of
Theorem \ref{thm:main_nonex}.

\begin{theorem}\label{thm:main_nonex_quant}
    Let $k$ be a positive integer.
    Let $q=p^h$ be a prime power and let $\cP$ be
    a polar space of rank $r$ and parameter $e$ over the field with $q$ elements, where  $r \geq k+1$.
    Let $\cO$ be a partial $(r, k)$-ovoid of $\cP$.
    Then
    \[
        |\cO| \leq \left(\prod_{i=1}^{k-1} \frac{(q^{r-i+1}-1)(q^{r+e-i}+1)}{(q^{i+1}-1)q^{2r+e-k-i}}  \right) \cdot (p+2r-2k+3)^{k h (p-1)}\leq 2^{k-1} (p+2r-2k+3)^{k h (p-1)}.
    \]
\end{theorem}

We will need a technical lemma for the proof.

\begin{lemma}\label{lem:lns_in_deg}
    Let $r \geq k+1$.
    Let $\cP$ be a finite classical polar space of rank $r$ with
    parameter $e$ naturally embedded in $V(n, q)$.
    Let $H$ be a degenerate hyperplane of $\cP$ in $V(n, q)$.
    Then the number of $k$-spaces of $\cP$ in $H$ is 
    \[
     \frac{q^{2r+e-k-1}+q^r-q^{r+e-1}-1}{(q^r-1)(q^{r+e-1}+1)} 
    \]
    of the total number of $k$-spaces in $\cP$.
\end{lemma}
\begin{proof}
 Using Lemma \ref{lem:numberofkinpolar}, we find that the number $k$-spaces in
 a rank $r$ polar space of type $e$ is \[\gauss{r}{k} \prod_{i=1}^k (q^{r+e-i}+1).\]

Note that since $H$ is a degenerate hyperplane, we know that $H\cap \cP$ is a cone with vertex $P=H^\perp$ and basis a polar space $\cP'_H$ with the same parameter $e$, and with rank $r-1$.

Now we calculate the number of $k$-spaces in  $H$.

We first count the number of $k$-spaces of $\cP$ in $H$ through $P$. By investigating the quotientspace of $P$, we find that this number is equal to the number of $(k-1)$-spaces in $\cP'_H$, and hence, is equal to 
\[\gauss{r-1}{k-1} \prod_{i=1}^{k-1} (q^{r-1+e-i}+1).\]

Now we count the number of $k$-spaces of $\cP$ in $H$ not though $P$. For this, we can project each of these $k$-spaces to the basis $\cP'_H$, and see that this number of $k$-spaces is equal to the number of $k$-spaces in $\cP'_H$ times the number of $k$-spaces in a $(k+1)$-space through $P$, but not containing $P$. 
This gives that this number of $k$-spaces of $\cP$ in $H$ not through $P$ is equal to
\[\gauss{r-1}{k} \prod_{i=1}^{k} (q^{r-1+e-i}+1)\cdot q^k.\]

This now implies that the number of $k$-spaces in $\cP\cap H$ is $X$ of the number of $k$-spaces in $\cP$ with 
\begin{align*}
    X&=\frac{q^k\gauss{r-1}{k} \prod_{i=1}^{k} (q^{r-1+e-i}+1)+\gauss{r-1}{k-1} \prod_{i=1}^{k-1} (q^{r-1+e-i}+1)}{\gauss{r}{k} \prod_{i=1}^k (q^{r+e-i}+1)}\\
    &=\frac{q^{2r+e-k-1}+q^r-q^{r+e-1}-1}{(q^r-1)(q^{r+e-1}+1)}.  \qedhere
\end{align*}
\end{proof}

\begin{proof}[Proof of Theorem \ref{thm:main_nonex_quant}]
    We will prove the result by induction.
    Theorem \ref{thm:blokhouse} shows the claim for $k=1$.


    Let $k \geq 2$.
    We can assume that the assertion is true for $(r-1, k-1)$.
    Recall that the collineation group of $\cP$ acts transitively on
    $k$-spaces of $\cP$, so each such $k$-space lies in the same number
    of degenerate hyperplanes.

    Hence, by Lemma \ref{lem:lns_in_deg}, we know that  a degenerate hyperplane 
     contains on average
     \[|\cO| \cdot \frac{q^{2r+e-k-1}+q^r-q^{r+e-1}-1}{(q^r-1)(q^{r+e-1}+1)}\] $k$-spaces of $\cO$. Let $H$ be such a degenerate hyperplane containing at least this number of $k$-spaces of $\cO$.

     The other elements of $\cO$, not contained in $H$ meet $H$ in a $(k-1)$-space. Let $\cO'$ be the set of all these $(k-1)$-spaces, and note that such a $(k-1)$-space cannot contain the vertex $H^\perp$.
    If we project all elements of $\cO'$ to the basis of the cone $H\cap \cP$, we see that $\cO'$ corresponds to a partial $(r-1, k-1)$-ovoid in the quotient
    of $H^\perp$ and we can apply our bound for these parameters.
    Furthermore, for some $L \in \cO'$,
    \[
        \{ K/L: K \in \cO, \, K \cap H = L \}
    \]
    is a partial $(r-k+1, 1)$-ovoid, that is a partial ovoid.
    Hence, by Lemma \ref{thm:blokhouse},
    each element of $\cO'$ lies in at most $(p+2r-2k+3)^{h(p-1)}$
    elements of $\cO$. Hence,
    
    \[
        |\cO| \leq |\cO'|(p+2r-2k+3)^{h(p-1)}+|\cO| \frac{q^{2r+e-k-1}+q^r-q^{r+e-1}-1}{(q^r-1)(q^{r+e-1}+1)},
    \]
    which implies that 
     \begin{align*}
         |\cO| &\leq \frac{(q^r-1)(q^{r+e-1}+1)}{q^{2r+e-k-1}(q^k-1)}|\cO'|(p+2r-2k+3)^{h(p-1)}\\
         &\leq \frac{(q^r-1)(q^{r+e-1}+1)}{q^{2r+e-k-1}(q^k-1)}\left(\prod_{i=1}^{k-2} \frac{(q^{r-i}-1)(q^{r-1+e-i}+1)}{(q^{i+1}-1)q^{2r-1+e-k-i}}  \right)\\ &\hspace{3cm} \cdot (p+2r-2k+3)^{(k-1) h (p-1)}(p+2r-2k+3)^{h(p-1)}
     \end{align*}
    
    This shows the claimed bound.
\end{proof}

For $p$ and $k$ fixed, the last bound in Theorem \ref{thm:main_nonex_quant}
is a polynomial in $r$. Lemma \ref{lem:nrinovoid} states
the size of an $(r, k)$-ovoid and it is an exponential function in $r$.
As an exponential function grows faster than any polynomial, this
shows Theorem \ref{thm:main_nonex}. Also note that when comparing 
the size of an $(r, k)$-ovoid with the given bound, then $h$
cancels on both sides. 

Lastly, let us note the following 
non-homogenous variant of Theorem \ref{thm:main_nonex}
which follows immediately from Theorem \ref{thm:main_nonex_quant}.
It can be interpreted as a variant of the observations 
in \cite{Irrsubcubepart} and \cite{AffinePaper}
that a tight irreducible subcube partition, respectively,
a tight irreducible affine vector space partition can only 
have a small number of subcubes, respectively, subspaces 
of large dimension (recall from Section \ref{sec:defs} that large dimensions in the hypercube 
setting correspond to small dimensions in polar spaces).

\begin{corollary}
    Let $k$ be a positive integer.
    Let $q=p^h$ be a prime power and let $\cP$ be
    a polar space of rank $r$ over the field with $q$ elements, where  $r \geq k+1$.
    Then for $k$, $p$ fixed, and $r \rightarrow \infty$, the proportion 
    of elements of $\cO$ of dimension at most $k$ is $o(1)$.
\end{corollary}

\section{Future Work}

Here we show that $(r, k)$-ovoids
are rare for $k$ small compared to $r$. It would be very interesting
to provide more concrete bounds, maybe even just for $k=2$.
Conversely, we show that there are plenty of examples
for $(r, r-1)$-ovoids in polar spaces with parameter $e \in \{ 0, 1/2, 1 \}$.
This suggests that for $r-k$ small, $(r, k)$-ovoids exist,
but we lack constructions. This is also true for the
non-homogenous case.

More generally, an $(r, k)$-ovoid covers each generator of a polar space
precisely once. Hence, it is a design in some sense and the following
question is natural: Can we find a family $\cD$ of $k$-spaces such that each
$t$-space contains precisely $\lambda$ elements of $\cD$?
Our question specializes to the case $(t, k, \lambda) = (r, k, 1)$.
The related existence question of covering $t$-spaces with $k$-spaces,
that is, if we can cover each generator with the same number of $k$-spaces,
has been recently answered in \cite{Weiss2023}
by Weiß.

%


\bibliographystyle{abbrv}

\end{document}